\documentclass[10pt,a4paper,twoside]{amsart}

\usepackage{amsfonts, amssymb, amsmath, amsthm, bm}
\usepackage{mathrsfs} % \mathscr
\usepackage{latexsym}%utilisation des symboles LaTeX pour avoir un beau LaTeX
\usepackage{enumerate}
\usepackage{multicol}
\usepackage{verbatim}
\usepackage{dsfont}
\usepackage{url}
\usepackage{csquotes}
\usepackage{placeins}
\usepackage{epic}
\usepackage{graphicx, subcaption}
\usepackage{epstopdf}
\usepackage[colorlinks=true]{hyperref}
\hypersetup{citecolor=blue, linkcolor=blue}
\usepackage{mathabx} %% for widecheck
\newtheorem{thm}{Theorem}[section]
\newtheorem{lem}[thm]{Lemma}
\newtheorem{cor}[thm]{Corollary}

\newcommand{\Z}[1]{\mathbb{Z}/#1\mathbb{Z}}

\DeclareMathOperator{\sgn}{sgn}

\def\Kcal{\mathcal{K}}

\def\Ocal{\mathcal{O}}

\def\1{\mathds{1}}

\def\K{\mathcal{K}}

\def\mode{\mathbin{\,\textrm{mod}^*}}
\def\mod{\mathbin{\,\textrm{mod}\,}}

\DeclareMathOperator{\sinc}{sinc}
\let\myalpha=a

\let\mybeta=b

\title{The large sieve through Parseval, Large sifted sets are regular%
}
\author{Olivier Ramar\'e}
\address[O. Ramar\'e]{CNRS/ Institut de Math\'ematiques de Marseille, Aix 
 Marseille Universit\'e, U.M.R. 7373, Site Sud, Campus de Luminy, Case 907, 
 13288 
 Marseille Cedex 9, France.}
\email{olivier.ramare@univ-amu.fr}

\begin{document}



\maketitle
\begin{abstract}
  %\texttt{File \jobname.tex} 
   We modify the approach to the arithmetical form of the large sieve
   by relying on the Parseval identity rather than on an approximate
   Bessel inequality and as a consequence, we prove that sifted sets
   are either small or the attached Fourier polynomial is very regular
   at the origin.
   We also discuss the optimality of this approach and shall show in
   passing that
   $N\int_{-\delta}^\delta|S(\beta)|^2d\beta\ge(1-2/\sqrt{\delta
     N})|S(0)|^2$ for any trigonometric polynomial of length~$N$.
 \end{abstract}

%{\small \tableofcontents}

%%%%%%%%%%%%%%%%%%%%%%%%%%%%%%%%%%%%%
\section{A lightning introduction to roughly assess the results}
%%%%%%%%%%%%%%%%%%%%%%%%%%%%%%%%%%%%%

Since their discovery by V.~Brun in~\cite{Brun*19, Brun*19b}, 
sieve methods have proved very efficient to \emph{bound above} quantities
defined by the exclusion of some or several residue classes. Sieve
methods are now classified in three main categories: the combinatorial
sieves, the Selberg (or $\Lambda^2$) sieve and the large
sieves. Gallagher's larger sieve introduced in~\cite{Gallagher*71}
remains an isolated tool, though a powerful one where it applies. We
are concerned in this paper with the large sieves. We only mention here that
these arose from the initial work of Yu.\,V.~Linnik
in~\cite{Linnik*42}; we refer the readers to the
monographs~\cite{Montgomery*71} by H.\,L.~Montgomery
and~\cite{Ramare*06} for more historical details.

%%%%%%%%%%%%%%%%%%%%%%%% 
%\subsection*{Two results for assessment}
%%%%%%%%%%%%%%%%%%%%%%%%
In the classical sieve setting, we start from a sequence of (excluded)
subsets $\mathcal{L}_p\subset\Z{p}$ for each prime $p\le Q$, for
some~$Q$, and consider the integers~$n$ from a host sequence, say an
interval~$[M+1,M+N]$, that are such that, for every $p\le Q$, we have
$n\notin\mathcal{L}_p$. Here is a first corollary of our method.
%%%%%%%%%%%%%%%
\begin{cor}
  \label{easycor}
    Let $\mathcal{Z}$ be the set of integers $n\in[M+1,M+N]$ that do not belong to any of
    the sets $\mathcal{L}_p$, for $p\le Q$.
    Choose two positive real parameters $Q\le M$ and~$V\le N/2$.  We have
    \begin{equation*}
      \frac{1-2V/N}{N}\Bigl|\sum_{\substack{n\in\mathcal{Z}}}1\Bigr|^2
      \le
      \int_{-\infty}^\infty\Bigl|\sum_{\substack{n\in\mathcal{Z}\\
          |n-t|\le V}}1\Bigr|^2\frac{dt}{4V^2}
      \le \frac{1+ Q^2/V}{L(Q)}\sum_{n\in\mathcal{Z}}1
  \end{equation*}
  where
  \begin{equation}
    \label{defL}
    L(Q) = \sum_{q\le Q}\mu^2(q)\prod_{p|q}\frac{|\mathcal{L}_p|}{p-|\mathcal{L}_p|}.
  \end{equation}
\end{cor}
%%%%%%%%%%%%%%%
This is to be compared with the classical large sieve bound which, in
this case, yields the upper bound
\begin{equation}
  \label{eq:170}
  \frac{1}{N}\Bigl|\sum_{\substack{n\in\mathcal{Z}}}1\Bigr|^2
  \le \frac{1+ Q^2/N}{L(Q)}\sum_{n\in\mathcal{Z}}1.
\end{equation}
The parameter~$V$ thus introduces a small loss, replacing in most
cases a~$1$ by a~$1+o(1)$, but, aside from the method which follows an
untrodden path, the cornered quantity is of interest. Here is a corollary.
%%%%%%%%%%%%%%
\begin{cor}
  \label{ManypsL2}
  Assume the number $Z$ of primes inside $[M+1,M+N]$ satisfies, for
  some positive $C$:
  \begin{equation*}
    Z\ge \frac{2N}{\log N+C\log \log N}.
  \end{equation*}
  Then, for any $A\ge2$, we have:
  \begin{equation*}    
    \int_{-(\log N)^A/N}^{(\log N)^A/N}\biggl|S(\beta)-\frac{Z}{N}
    \sum_{n\le N} e(n\beta)\biggr|^2d\beta
    \ll_{A,C}
    \frac{Z\log\log N}{(\log N)^2}.
  \end{equation*}
\end{cor}
%%%%%%%%%%%%%%
A more general version is stated in Corollary~\ref{ManypsL20}, while
an $L^\infty$-version is proved in Corollary~\ref{Manyps}.

We recall that we know that, see \cite{Ramare-SchlagePuchta*06}, that
$Z\le 2N/(3.53+ \log N)$ provided~$N$ is large enough. One
would like to increase this constant~3.53, and possibly show that it
may be taken infinitely large. The above corollary tells us that, if
this quantity $3.53$ remains smaller than $C\log\log N$, then the
associated trigonometric polynomial is very regular on a
long interval around the origin. We propose in Corollary~\ref{Manyps}
an $L^\infty$-version of this bound and in Corollary~\ref{Manypsq} a
version for primes in arithmetic progressions.

Improving on the inequality on $Z$ may have
consequences on the location of the possible Siegel zero, see for
instance~\cite{Motohashi*79} by Y.\,Motohashi or
\cite{Ramachandra-Sankaranarayanan-Srinivas*96} by
K.~Ramachandra, A.~Sankaranarayanan \&{} K.~Srinivas.

%%%%%%%%%%%%%%%%%%%%%%%%%%%%%%%%%%%%%
\section{Introduction and results}
%%%%%%%%%%%%%%%%%%%%%%%%%%%%%%%%%%%%%
Let us now proceed more slowly.

%%%%%%%%%%%%%%%%%%%%%%%%
\subsection*{A short roadmap on the Large Sieve}
%%%%%%%%%%%%%%%%%%%%%%%%
When using the large sieve inequality for sieving purpose, whether as
Linnik originally did or in the modern version that is
%\index{Linnik@Linnik,~Yu.\,V.}%
Montgomery's sieve, we rely on
two pieces of informations:
\begin{itemize}
\item Some \emph{local lower bounds} of arithmetical nature, see
  Theorem~\ref{L2S} or~\eqref{eq:6} below,
\item and a
\emph{global upper bound}, usually the
Large Sieve inequality, see~\eqref{eq:5} below. 
\end{itemize}
When applied to the primes-in-interval situation, we first define
\begin{equation}
  \label{defSinterval}
  S(\alpha)=\sum_{M<p\le M+N}e(p\alpha),
  \quad (e(\beta)=\exp 2i\pi\beta)
\end{equation}
where $p$ denotes a prime number and $M$ and $N\ge1$ are real
numbers. We readily find that, when $q\le M$, we have
\begin{equation}
  \label{eq:6}
  \sum_{a\mode q}|S(a/q)|^2\ge \frac{\mu^2(q)}{\varphi(q)}|S(0)|^2
\end{equation}
(the notation $a\mode q$ denotes a summation over every reduced
residue classes $a$ modulo~$q$)
and the Large Sieve inequality (see Theorem~1
in~\cite{Montgomery-Vaughan*73}) tells us that
\begin{equation}
  \label{eq:5}
  \sum_{q\le Q}\sum_{a\mode q}|S(a/q)|^2\le (N+Q^2)S(0).
\end{equation}
When we join both, we swiftly infer that
\begin{equation*}
  (\log Q) |S(0)|^2
  \le
  \sum_{q\le Q}\frac{\mu^2(q)}{\varphi(q)}|S(0)|^2
  \le (N+Q^2)S(0)
\end{equation*}
from which we deduce that the number $S(0)$ of primes in the interval
$(M,M+N]$ is at most\footnote{We selected
$Q=\sqrt{N}/\log N$ and assumed that $Q\le M$, an
assumption that is easily lifted.} $(2+o(1))N/\log N$. This is a
simple version of the celebrated Brun-Titchmarsh Inequality (see
Theorem~2 in~\cite{Montgomery-Vaughan*73}).

A distinctive feature of the Large Sieve situation, with respect to sieves in
general is that we sieve out intervals and not some general host
sequence, typically a sequence of polynomial values.

%%%%%%%%%%%%%%%%%%%%%%%%
\subsection*{On the method}
%%%%%%%%%%%%%%%%%%%%%%%%
We rely on (a slightly more general but essentially) the same local
lower estimate, but modify the global inequality. A first path goes
through an approximated Bessel inequality, like Proposition~1
in~\cite{Bombieri*71} by E.\,Bombieri. Montgomery \& Vaughan
in~\cite{Montgomery-Vaughan*73} use a (weighted or not ) version of
Hilbert's inequality. We refer the readers to \cite{Yangjit*23} by
W.\,Yangjit \index{Yangjit@Yangjit,~Wijit} and to
\cite{Carneiro-Littman*24} by E.\,Carneiro \& F.\,Littmann for recent
work on this inequality. Instead of that, we rely on the Farey
dissection of the unit interval, as in the circle method; our
global inequality is a consequence of the Parseval Identity. It is
recalled in Lemma~\ref{Step2b}.

%%%%%%%%%%%%%%%%%%%%%%%%
\subsection*{A more general result}
%%%%%%%%%%%%%%%%%%%%%%%%
We presented sieving in the first paragraph through an exclusion
hypothesis. This approach, though natural and historically older, 
leads to difficulties. It is better to say that
$n\in\Kcal_p=\Z{p}\setminus \mathcal{L}_p$.
This is for instance the
viewpoint adopted by M.\,N.~Huxley in his book~\cite{Huxley*72-2} and
by the present author in~\cite{Ramare*06, Ramare*10, Ramana-Ramare*25}.

Let us present the sieving situation from scratch.  A subset
$\K_q\subset\Z{q}$ is said to be \emph{multiplicative}\footnote{In
  earlier work, I used \emph{multiplivatively split} instead of the
  simpler \emph{multiplicative}.} if, when the decomposition of $q$ in
prime factors reads
\begin{equation*}
  q=p_1^{e_1}p_2^{e_2}\cdots p_r^{e_r},\quad (\forall i\neq j, \quad p_i\neq p_j),
\end{equation*}
and the Chinese Remainder Map is defined by
\begin{equation*}
  %\label{eq:7}
  \sigma:
  \begin{array}[t]{rcl}
    \Z{q}&\rightarrow&{\displaystyle \prod_{1\le i\le r}\Z{p_i^{e_i}}}\\
    x&\mapsto&\bigl(x\mod p_i^{e_i}\bigr),
  \end{array}
\end{equation*}
we have the property
$%\begin{equation}
  %\label{eq:12}
  \sigma^{-1}\bigl(\sigma\bigl(\K_q\bigr)\bigr)=\K_q$.
%\end{equation}
This is often written in the shorter form
%\begin{equation}
%  \label{eq:13}
 $ \K_q=\prod_{1\le i\le r}\K_{p_i^{e_i}}$.
%\end{equation}
We further say that the sequence $(\Kcal_q)_{q\le Q}$ is
\emph{consistent} when $\Kcal_q/d\mathbb{Z}=\Kcal_d$ whenever $d|q$.
It may be expedient to restrict the modulus~$q$ to square-free values,
it would solely be more difficult to write, as we should write
\begin{equation*}
  (\Kcal_q)_{q\le Q, \text{$q$ square-free}}.
\end{equation*}
We finally say that the \emph{Johnsen-Gallagher condition} 
  holds
  whenever
  \begin{equation}
    \label{jonhsen}
    \forall d|q,
    \forall y\in\K_d,\quad\#\{x\in\K_q:x\equiv y[d\}=|\K_q|/|\K_d|.
  \end{equation}
  This is equivalent to saying that the number of preimages in $\K_q$
  of any point~$y$ of~$\K_d$ does not depend on~$y$.
When $q$ is square-free and $\K_q$ is multiplicative, this condition
always holds.

Given a consistent multiplicative sequence $(\Kcal_q)_{q\le Q}$, we define
\begin{equation}
  \label{defg}
  g(q)=\prod_{p^\alpha\|q}
  \biggl(\frac{p^\alpha}{|K_{p^\alpha}|}
  -\frac{p^{\alpha-1}}{|K_{p^{\alpha-1}}|}\biggr).
\end{equation}

A sequence $(u_n)_{M<n\le M+N}$ of complex numbers is said \emph{have
support on $(\Kcal_q)_{q\le Q}$} whenever, when $u_n\neq0$, then $n$
belongs to every $\Kcal_q$ for $q\le Q$.
We may now state our main theorem.
%%%%%%%%%%%%%
\begin{thm}[A localized weighted sieve]
  \label{WeightedSieve}
  Let $(u_n)_{M<n\le M+N}$ be a sequence of non-negative real numbers
  having support on a consistent multiplicative sequence
  $(\Kcal_q)_{q\le Q}$ that satisfies the Johnsen-Gallagher
  condition. For any parameter $V\le N/2$, we have
  \begin{equation*}
    \frac{1-2V/N}{N}\Bigl|\sum_{\substack{M<n\le M+N}}u_n\Bigr|^2
    \le
    \int_{-\infty}^\infty\Bigl|\sum_{\substack{M<n\le M+N\\|n-t|\le V}}u_n\Bigr|^2\frac{dt}{4V^2}
    \le 2V
    \sum_{n\le N}u_n^2/L^*(Q,V)
  \end{equation*}
  where
  \begin{equation}
    \label{defLstar}
    L^*(Q,V) = \sum_{\substack{q\le Q}}
    \frac{g(q)}{2V+q(q+Q)}.
  \end{equation}
  The function $g$ is defined at~\eqref{defg}.
\end{thm}
%%%%%%%%%%%%%
Theorem~\ref{Nearby} proposes a very close version of the same result that
may be more accessible.
This is to be compared with the weighted sieve proposed by H.~Montgomery \& R.\,C.~Vaughan
in Corollary~1
of~\cite{Montgomery-Vaughan*73} which reads in our notation
\begin{equation}
  \label{eq:4}
  \Bigl|\sum_{\substack{M<n\le M+N}}u_n\Bigr|^2
  \le
  \sum_{n\le N}u_n^2/L^{**}(Q, N)
\end{equation}
where
\begin{equation}
  \label{eq:14}
  L^{**}(Q,N) = \sum_{\substack{q\le Q}}
  \frac{g(q)}{N+\frac32 qQ}.
\end{equation}
The coefficient $3/2$ may even be replaced by
$\sqrt{1+\frac23\sqrt{6/5}}=1.315\dots$  after the work of E.~Preissmann in~\cite{Preissmann*84},
and this appears to be the
last improvement on this matter. The relative strength of our result
lies in the localization property, i.e. on the parameter~$V$ which can
be chosen typically of size~$N/(\log N)^{A}$.

Corollary~\ref{easycor} is an immediate consequence of this result.
We presented the sieving situation by using consistent sequences and
the Johnsen-Gallagher condition. An approach closer to the one
employed by~A.~Selberg in~\cite{Selberg*76} is proposed
in~\cite{Jha-Ramana-Ramare*25}. Notice that the non-negativity is
superfluous as we may apply the result to~$(|u_n|)$ rather than to
$(u_n)$. We can even allow complex values.

%%%%%%%%%%%%%% 
\subsection*{Nearby values}
%%%%%%%%%%%%%%
As announced at the beginning of this section, we may exploit the local
character of Theorem~\ref{WeightedSieve}. We start with an
$L^\infty$-bound and continue with an $L^2$-bound.
%%%%%%%%%%%%%%%%%%
\begin{thm}
  \label{conditional1}
  When $(u_m)$ is as in Theorem~\ref{WeightedSieve}, $S(0)>0$ and
  $V<N/4$, and under the assumption
  \begin{equation*}
    2S(0)^2\ge N\|S\|_2^2/L(Q)
  \end{equation*}
  where $L$ is given in~\eqref{defL} (or any positive lower bound of it), we have
  \begin{equation*}
    \biggl|\frac{S(\beta)}{S(0)}
    -
    \int_{-V}^{N+V}\mkern-15mu e(\beta t)\frac{dt}{N}
    \biggr|
    \le \sqrt{\frac{N\|S\|_2^2/L(Q)}{S(0)^2}-1}
    +\frac{2Q}{\sqrt{V}}+7 |\beta| V.
  \end{equation*}
\end{thm}
%%%%%%%%%%%%%%%%%%
The approximation is rather weak, but the range in $\beta$ is
very large. 

It is worth trying this result on a typical problem.
%%%%%%%%%%%%%%
\begin{cor}
  \label{Manyps}
  Assume the number $Z$ of primes inside $[M+1,M+N]$ satisfies, for
  some positive $C$:
  \begin{equation*}
    Z\ge \frac{2N}{\log N+C\log \log N}.
  \end{equation*}
  Then, for any $B\ge2$, we have for $N\gg e^{4B}/B$:
  \begin{equation*}
    \biggl|\sum_{\substack{M+1\le p\le M+N}}\mkern-15mu e(p\beta)
    -
    2\int_{0}^{N} \frac{e(\beta t)dt}{\log N}
    \biggr|
    \le 2Z \sqrt{\frac{(B+C)\log\log N}{\log N}}
  \end{equation*}
  uniformly for every $\beta$ such that $N|\beta|\le (\log N)^{B}$.
\end{cor}
%%%%%%%%%%%%%%

Let us specify that the implied constant in $N\gg e^{4B}/B$ is also
independent of~$C$.
We stated these results for primes in interval, for clarity without any additional
congruence condition attached to these primes, but the same results
hold in the more general setting. Let us record such a statement.
%%%%%%%%%%%%%%
\begin{cor}
  \label{Manypsq}
  Let $q\ge1$ be a modulus and $a$ be coprime to~$q$.
  Let $N>q$ be given.
  Assume the number $Z$ of primes confurent to $a$ modulo~$q$ inside
  $[M+1,M+N]$ satisfies, for 
  some positive $C$:
  \begin{equation*}
    Z\ge \frac{2N}{\varphi(q)(\log (N/q)+C\log \log (N/q)}.
  \end{equation*}
  Then, for any $B\ge2$, we have for $N/q\gg e^{4B}/B$:
  \begin{equation*}
    \biggl|\sum_{\substack{M+1\le p\le M+N\\ p\equiv a[q]}}\mkern-15mu e(p\beta)
    -
    2\int_{0}^{N} \frac{e(\beta t)dt}{\log N}
    \biggr|
    \le 2Z \sqrt{\frac{(B+C)\log\log (N/q)}{\log (N/q)}}
  \end{equation*}
  uniformly for every $\beta$ such that $(N/q)|\beta|\le \log^B (N/q)$.
\end{cor}
%%%%%%%%%%%%%%
Here is another expression of the same phenomenom.
%%%%%%%%%%%%%
\begin{thm}
  \label{Nearby}
  Let $(u_n)_{M<n\le M+N}$ be a sequence of non-negative real numbers
  having support on a consistent multiplicative sequence
  $(\Kcal_q)_{q\le Q}$ that satisfies the Johnsen-Gallagher
  condition. For any parameter $V\le N/2$ and such that $V-\tfrac12$
  is an integer, we have 
  \begin{multline*}
    \int_{-1/2}^{1/2}\biggl|S(\beta)-\frac{S(0)}{N}
    \sum_{-2V< n< N+2V}\mkern-20mu e(n\beta)\biggr|^2
    \biggl|\frac{\sin \pi (2V-1)\beta}{2V\sin\pi\beta}\biggr|^2d\beta
    \\+\frac{|S(0)|^2}{N}\biggl(1-\frac{2V}{N}\biggr)
    \le \sum_{n}u_n^2/(2VL^*(Q,V))
  \end{multline*}
  where $L^*$ is defined in~\eqref{defLstar} and
  $S(\beta)=\sum_n u_n e(n\beta)$.
\end{thm}
%%%%%%%%%%%%%
This result gives access to the deviation of $S(\beta)$
from~$\frac{S(0)}{N} \sum_{n} e(n\beta)$  in $L^2$-norm while
Theorem~\ref{conditional1} gives access to single point deviation.
Here is a simpler version aimed at applications.
%%%%%%%%%%%%%
\begin{cor}
  \label{ManypsL20}
  Let $Z$ be the number of points $n$ that belong to every $\Kcal_q$,
  where $(\Kcal_q)_{q\le Q}$ is
  a consistent multiplicative sequence
  that satisfies the Johnsen-Gallagher
  condition. Assume that $L(Q)$ satisfies
  $L(Q)=C (\log Q)^\kappa(1+\Ocal(1/\log Q))$ when $Q\ge2$, for some
  $C>0$ and $\kappa\ge0$.
  For any parameter $A\ge1$, we have 
  \begin{equation*}
    \int_{-(\log N)^A/N}^{(\log N)^A/N}\biggl|S(\beta)-\frac{Z}{N}
    \sum_{n\le N} e(n\beta)\biggr|^2d\beta
    %\\
    \ll_A
    \biggl(\frac{N}{L(\sqrt{N})}-Z
    \biggr)\frac{1}{(\log N)^{\kappa}}+
    \frac{N\log\log N}{(\log N)^{2\kappa+1}}
  \end{equation*}
  where $S(\beta)=\sum_n e(n\beta)$.
\end{cor}
%%%%%%%%%%%%%
The condition on $L(Q)$ is usually replaced by saying
that ``the sieve is of dimension $\kappa$'', by which we mean that
we have
\begin{equation}
  \label{eq:19}
  %\left\{
  %    \begin{aligned}
        \sum_{p^k\le P}g(p^k)\log(p^k)=\kappa\log P+\Ocal(1),\quad
        \sum_{p, k,\nu\ge2}g(p^k)g(p^\nu)\log(p^k)\le A.
   %   \end{aligned}
   % \right.
\end{equation}
Under such hypotheses, we have, when $Q\ge2$,
\begin{equation*}
  L(Q)=\sum_{q\le Q}g(q)=C(\log Q)^\kappa(1+\Ocal(1/\log Q))
\end{equation*}
for some positive constant~$C$ that can be fully described.
This result can be found in Theorem~21.1
of~\cite{Ramare*06}, but, when $g(p^k)=0$ as soon as $k\ge2$, this
result is much older and can for instance be found for instance
in~\cite{Halberstam-Richert*71} by H. Halberstam \& {H.E.} Richert.

%%%%%%%%%%%%%% 
\subsection*{A problem in Harmonic Analysis}
%%%%%%%%%%%%%%
The reader will find in Subsection~\ref{HA} a problem of Harmonic
Analysis on which we stumbled and which has some independent interest.
%%%%%%%%%%%%%% 
\subsection*{Acknowledgement}
%%%%%%%%%%%%%%
The present paper was started when the author was enjoying the
hospitality of the Anhui university in China, and completed when the
author was enjoying the hospitality of the International Centre for
Theoretical Sciences (ICTS) in Bengaluru, India during the program - The
Classical Circle Method and the Large Sieve 2026 (code:
ICTS/CCMLSS2026/05).
Most of the computations have been run by Pari/GP~\cite{PARI-GP} and,
for the figures, either the data in \texttt{svg}-format have been
piped to Inskscape~\cite{Inkscape} for 
functions, or to Sage~\cite{sagemath} to obtain the plots.

%%%%%%%%%%%%%%%%%%%%%%%%%%%%%%%%%%%%%%%
\section{A local lower bound}
%%%%%%%%%%%%%%%%%%%%%%%%%%%%%%%%%%%%%%%
This section is devoted to proving the next lower bound. This theorem
is new only in its generality. See for instance the
proof of Theorem~6 in the book~\cite{Bombieri*74} by E.~Bombieri or
Eq.\,(8.8) in the book~\cite{Huxley*72-2} by M.\,N.~Huxley.
The proof we give may have further
consequences, in particular Identity~\eqref{L1S} which is an important
ingredient in the proof of Theorem~1.4 in~\cite{Ramare*25-1}. 

%%%%%%%%%%%%%%%%%%%%% 
\begin{thm}
  \label{L2S}
  \label{L2LB}
  When $(u_n)$ has support on the consistent and multiplicative
  sequence~$(\Kcal_d)_{d|q}$ which satisfies the Johnsen-Gallagher
  condition, and with $S(\alpha)=\sum_n
  u_ne(n\alpha)$, we have
  \begin{equation*}
  \label{eq:31}
  \sum_{a\mode q}|S(a/q)|^2
  \ge
  g(q)
  |S(0)|^2.
\end{equation*}
\end{thm}
%%%%%%%%%%%%%%%%%%%%%%
We provide two proofs of this result; the first one is restricted
to the case of the primes but its sheer simplicity should be enlightening
to the readers. Notice that this inequality may be expected to be
optimal. It is so at least in the case of the primes from the initial
segment (i.e. from $[1,N]$) and when $q$ is not too large. See
Theorems~3 and~4 of \cite{Liu-Zhan*97} by Liu, Jianya \& Zhan, Tao
again concerning the case of the primes from the initial segment but
for moduli~$q$ roughly up to $N^{1/3}$.
We propose Figures~\ref{test1png} and~\ref{test2png} on the
behaviour when the modulus~$q$ becomes of the size~$\sqrt{N}$.%
% %%%%%%%%%%%%
% \begin{figure}[!h]
%   \centering%
%   \hspace*{-7pt}%%
%   \subfloat{\includegraphics[width=.49\linewidth]{../TestingL2B-2-1000000.png}}%%
%   \qquad
%   \subfloat{\includegraphics[width=.49\linewidth]{../TestingL2B-2-10000000.png}}%%
%   \hspace*{-10pt}%
%   \caption{On the optimality of Theorem~\ref{L2S} for the primes}%
%   \label{test1png}%
% \end{figure}
% %%%%%%%%%%%%%
%%%%%%%%%%%%
\begin{figure}[!h]
  \includegraphics[scale=0.6]{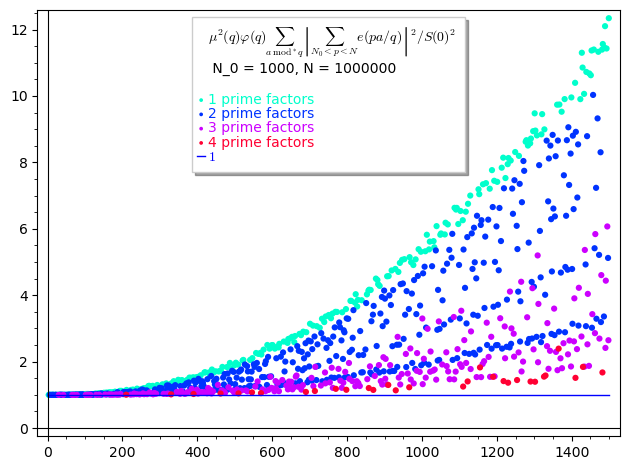}%%
  \caption{On the optimality of Theorem~\ref{L2S} for the primes}%
  \label{test1png}%
\end{figure}
%%%%%%%%%%%%%
%%%%%%%%%%%%
\begin{figure}[!h]
  \includegraphics[scale=0.6]{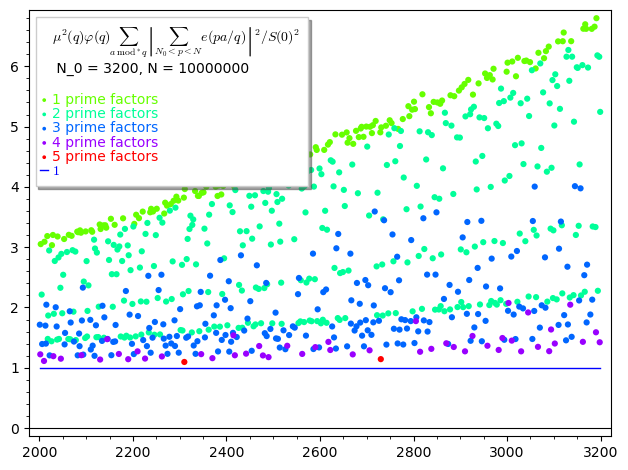}%%
  \caption{On the optimality of Theorem~\ref{L2S} for the primes}%
  \label{test2png}%
\end{figure}
%%%%%%%%%%%%%
We should however warn the readers against too fast deductions from
these plots: in many situations, the behaviour of the small primes
does not reflect the generic situation. Some regular patterns seem
nonetheless to appear, including some accumulation
directions. Furthermore the less prime factors has the modulus~$q$,
the larger the sum we consider. This
phenomenom will appear more strongly in Figures~\ref{test1twinpng}
and~\ref{test2twinpng} displayed at the end of this paper and that
concern the prime twin case.
%%%%%%%%%% 
\begin{proof}[Proof of Theorem~\ref{L2S} for the primes]
  Let us recall that the Ramanujan sum
  \begin{equation*}
    c_q(n)=\sum_{a\mode q}e(na/q)
  \end{equation*}
  is such that $c_q(n)=\mu(q)$ when $(n,q)=\gcd(n,q)=1$. Therefore
  \begin{equation*}
    \sum_{a\mode q}S(a/q)
    =\sum_{n}u_n\sum_{a\mode q}e(na/q)
    =\sum_{n}u_n\,\mu(q)=\mu(q)S(0)
  \end{equation*}
  by our hypothesis on $(u_n)$.
  Using Cauchy's inequality on this identity leads to
  \begin{equation*}
    |S(0)|^2\mu^2(q)\le \varphi(q)\sum_{a\mode q}|S(a/q)|^2
  \end{equation*}
  and the claimed inequality is proved since, in this case, we have
  $g(q)=\mu^2(q)/\varphi(q)$. 
\end{proof}
%%%%%%%%%%

%%%%%%%%%%%% 
\begin{proof}[Proof of Theorem~\ref{L2S} in general]
Let us define, as per Eq.\,(11.13) from~\cite{Ramare*06}, the function:
\begin{equation}
  \label{defpsistarq}
  \psi^*_q(n)=\sum_{\substack{\delta|q\\ n\in\Kcal_\delta}}\mu(q/\delta)
  \frac{\delta}{|\Kcal_\delta|}
\end{equation}
so that
\begin{equation}
  \label{eq:27}
  \sum_{\delta|q}\psi^*_\delta(n)=\psi_q(n)=\frac{q}{|\Kcal_q|}\1_{\Kcal_q}(n).
\end{equation}
The function $\psi_q$ is the \emph{local model modulo~$q$} for our
sequence and $\psi^*_q$ is its ``new'' part.
Notice the following important property:
%%%%%%%%%%%
\begin{center}
  \sl When $n\in\Kcal_q$, we have
  $\displaystyle\psi^*_q(n)=
  %(-1)^{\omega(q)}
  g(q)$.
\end{center}
%%%%%%%%%%% 
Let us compute the Fourier decomposition modulo~$q$ of $\psi_q$. We
find that
\begin{align*}
  \psi_q(n)
  &=
    \frac{1}{q}\sum_{\delta|q}\frac{q}{|\Kcal_q|}\sum_{a\mode \delta}
    \sum_{b\in\Kcal_q}e(ab/\delta)e(-na/\delta)
  \\&=
    \sum_{\delta|q}\frac{1}{|\Kcal_\delta|}\sum_{a\mode \delta}
    \sum_{c\in\Kcal_\delta}e(ac/\delta)e(-na/\delta).
\end{align*}
By identification, or by using the Moebius inversion formula, we infer
that
\begin{equation}
  \label{eq:28}
  \psi^*_q(n)
  =
  \sum_{a\mode q}
  \hat{\psi}^*_q(a)e(-na/q)
  \quad\text{where}\quad
  \hat{\psi}^*_q(a)
  =\frac{1}{|\Kcal_q|}
  \sum_{c\in\Kcal_q}e(ac/q).
\end{equation}
This quantity computation may also be found in the paper~\cite{Kobayashi*73} by
I.\,Kobayashi (this is~$b_{q,a}$ therein).
Consequently,  we find that 
  \begin{equation}
    \label{L1S}
    \sum_{a\mode q}\overline{\hat{\psi}^*_q(a)}S(a/q)
    =
    g(q)
    S(0).
  \end{equation}
On reading carefully Chapter~8, and precisely Eq\,(8.4) therein, of
the book~\cite{Huxley*72-2} by M.\,N.~Huxley, the reader will discover
the very same (fundamental) identity.

We compute furthermore:
\begin{align*}
  \sum_{a\mode q}|\hat{\psi}^*_q(a)|^2
  &=
  \frac{1}{|\Kcal_q|^2}\sum_{b_1,b_2\in\Kcal_q}c_q(b_1-b_2)  
  =
  \frac{1}{|\Kcal_q|^2}\sum_{b_1,b_2\in\Kcal_q}\sum_{\substack{d|b_1-b_2\\
  d|q}}d\mu(q/d)
  \\&=
  \frac{1}{|\Kcal_q|^2}\sum_{d|q}d\mu(q/d)\frac{|\Kcal_q|^2}{|\Kcal_d|^2}|\Kcal_d|
  =g(q).
\end{align*}
Thus, using Cauchy's inequality on~\eqref{L1S} and on using the above
$L^2$-norm computation, we reach Theorem~\ref{L2S}.
\end{proof}

%%%%%%%%%%%%%%%%%%%%%%%%%%%%%%%%%%%%%%% 
\section{A large sieve inequality alternative}
%%%%%%%%%%%%%%%%%%%%%%%%%%%%%%%%%%%%%%%
The large sieve inequality for the Farey sequence states that
\begin{equation}
  \label{LSFarey}
  \sum_{q\le Q}\sum_{a\mode q}|S(a/q)|^2\le \sum_n |u_n|^2(N+Q^2).
\end{equation}
This inequality joined with Lemma~\ref{L2LB} forms the gives the
bedrock of a proof of the
Brun-Titchmarsh inequality.
Let us start with a consequence of the Parseval equality that has a
similar flavour.
%%%%%%%%%%%%
\begin{lem}
  \label{Step2b}
  We have
  \begin{equation*}
    \sum_{q\le Q}\sum_{a\mode q}
    \int_{-\frac{1}{q(q+Q)}}^{\frac{1}{q(q+Q)}}\biggl|
    S\biggl(\frac{a}{q}+\beta\biggr)
    \biggr|^2d\beta
    \le \sum_{n\le N}|u_n|^2,
  \end{equation*}
   the parameter $Q\ge1$ being real number.
%  When $(u_n)$ is carried by integers prime to $P(Q+1)$, we have
\end{lem}
%%%%%%%%%%%%
This inequality is immediate for anyone used to the Circle
Method.
Notice the absence of the factor $N+Q^2$ that appears
in~\eqref{LSFarey}. The price to pay is to have an integral
around~$a/q$. As a comparaison, let us mention the inequality
\begin{equation}
    \sum_{q\le Q}\sum_{a\mode q}
    \sum_{\substack{\ell{}\in\mathbb{Z}\\ 2q|\ell{}|\le Q}}
    \biggl|
    S\biggl(\frac{a}{q}+\frac{\ell{}}{Q^2}\biggr)
    \biggr|^2
    \le \sum_{n\le N}|u_n|^2(N+Q^2).
  \end{equation}
proved in Theorem~1.8 of~\cite{Ramana-Ramare*25}.
%%%%%%%%%%%
\begin{proof}
  We use the Parseval identity together with the subsets
  \begin{equation}
    \label{eq:3}
    I_Q(a/q)=\frac{a}{q}+\biggl[\frac{-1}{q(q+Q)},\frac{1}{q(q+Q)}\biggr],
    q\le Q.
  \end{equation}
  These subsets are disjoint as, with an obvious notation,
  \begin{align*}
    \biggl|\frac{a}{q}+\beta-\frac{\tilde a}{\tilde q}-\tilde \beta\biggr|
    &> \frac{|a\tilde q-\tilde aq|}{q\tilde
      q}-\frac{1}{q(q+Q)}-\frac{1}{\tilde q(\tilde q+Q)}
    \\&> \frac{1}{q\tilde
      q}-\frac{1}{q(q+\tilde q)}-\frac{1}{\tilde q(\tilde q+q)}=0.
  \end{align*}
  The Kloosterman arcs decomposition would also leads to a proof of our lemma.
\end{proof}
%%%%%%%%%%%

%%%%%%%%%%%%%%%%%%%%%%%%%%%%%%%%%%%%%%% 
\section{Norm decomposition}
%%%%%%%%%%%%%%%%%%%%%%%%%%%%%%%%%%%%%%%
We prove in this preliminary section a lemma that will be used later on.

%%%%%%%%%%%
\begin{lem}
  \label{integraldecomposition}
  When the function $\myalpha$ is Riemann-integrable and supported on $[-V,V]$,
  and $(u_m)_{m\le N}$ are complex numbers, we have 
  \begin{multline*}
    \int_{-\infty}^\infty
    \Bigl|\sum_{m}u_m\myalpha(t-m)\Bigr|^2
    dt
    =
    \frac{|\Sigma|^2}{N}
    \biggl(1-\frac{2V}{N}\biggr)\Bigr|\sum_m u_m\Bigl|^2
    \\+\int_{-V}^{N+V}
    \biggl|\sum_{m}u_m\myalpha(t-m)-\frac{\Sigma\sum_m u_m}{N}\biggr|^2
    dt.
  \end{multline*}
  where $\Sigma=\int_{-V}^V\myalpha(t)dt$.
\end{lem}
%%%%%%%%%%%

%%%%%%%%%%%
\begin{proof}
  Let us set $X=\sum_m u_m$.
  We first find that
  \begin{align*}
    \int_{-V}^{N+V}
    \sum_{m}u_m\myalpha(t-m)dt
    &=
      \sum_m u_m
      \int_{-V}^{N+V}
      \myalpha(t-m)dt=\Sigma X.
  \end{align*}
  Therefore
  \begin{multline*}
    \int_{-V}^{N+V}
    \biggl|\sum_{m}u_m\myalpha(t-m)-\frac{\Sigma X}{N}\biggr|^2
    dt
    \\=
      \int_{-V}^{N+V}
    \biggl|\sum_{m\le V}u_m\myalpha(t-m)\biggr|^2
    dt
      -\frac{2\overline{\Sigma}}{N}\Re \overline{X}\int_{-V}^{N+V}
    \sum_{m}u_m\myalpha(t-m)dt
    +
      \frac{|\Sigma|^2(N+2V)}{N^2}|X|^2
    \\=
    \int_{-V}^{N+V}
    \biggl|\sum_{m}u_m\myalpha(t-m)\biggr|^2
    dt-\frac{|\Sigma|^2}{N}|X|^2
    +\frac{2V|\Sigma|^2|X|^2}{N^2}.
  \end{multline*}
  The lemma follows readily.
\end{proof}
%%%%%%%%%%%

%%%%%%%%%%%%%%%
\begin{lem}
  \label{Switch}
  Let $(u_m)$ be a finite complex sequence which we extend by~0 when
  the index lies outside the initial range of definition. 
  When $V-\tfrac12$ is an integer, we have
  \begin{equation*}
    \int_{-\infty}^\infty
    \Bigl|\sum_{m}u_m\myalpha_0(t-m)\Bigr|^2
    dt
    =
    \int_{-1/2}^{1/2}|S(\beta)|^2
    \biggl|\frac{\sin \pi (2V-1)\beta}{\sin\pi\beta}\biggr|^2d\beta.
  \end{equation*}
\end{lem}
%%%%%%%%%%%%%%%

%%%%%%%%%%%%%%
\begin{proof}
  We associate to $(u_m)$ the distribution
  \begin{equation}
    \label{defU}
    U(t)=\sum_{m\in\mathbb{Z}}u_m\delta_m
  \end{equation}
  so that
  \begin{equation}
    \label{expandUstara}
    (U\star \myalpha)(t)=\int_{\infty}^\infty U(y)\myalpha(t-y)dy
    =\sum_{m\in\mathbb{Z}}u_m\myalpha(t-m).
  \end{equation}
  We may directly check from the expression in~\eqref{expandUstara} that
  \begin{equation*}
    \widehat{U\star \myalpha}(x)=S(x)\hat{\myalpha}(x).
  \end{equation*}
  By Parseval, this gives us that
  \begin{equation}
    \label{eq:5}
    \int_{-\infty}^{\infty}|(U\star\myalpha)(t)|^2dt
    =
    \int_{-\infty}^{\infty}|S(x)|^2|\hat{\myalpha}(x)|^2dx.
  \end{equation}
  We also find that
  \begin{equation}
    \label{eq:6}
    \hat{\myalpha}_0(x)
    =\frac{\sin 2\pi Vx}{\pi x}.
  \end{equation}
  Consequently
  \begin{equation*}
    \int_{-\infty}^\infty
    \Bigl|\sum_{m}u_m\myalpha_0(t-m)\Bigr|^2
    dt
    =\int_{-\infty}^{\infty}|S(x)|^2\biggl|\frac{\sin 2\pi Vx}{\pi x}\biggr|^2dx.
  \end{equation*}
  As $S(x)$ is 1-periodical, we find that
  \begin{equation*}
    \int_{-\infty}^{\infty}|S(x)|^2\biggl|\frac{\sin 2\pi Vx}{\pi
      x}\biggr|^2dx
    =
    \int_{-1/2}^{1/2}
    |S(x)|^2\sum_{m\in\mathbb{Z}}\biggl(\frac{\sin 2\pi V(x-m)}{\pi (x-m)}\biggr)^2dx.
  \end{equation*}
  Let us set
  \begin{equation}
    \label{defFV}
    F_V(x)=\biggl(\frac{\sin 2\pi Vx}{\pi x}\biggr)^2.
  \end{equation}
  By Poisson summation formula, we have
  \begin{equation}
    \label{eq:12}
    \sum_{m\in\mathbb{Z}}\biggl(\frac{\sin 2\pi V(x-m)}{\pi
      (x-m)}\biggr)^2
    =
    \sum_{\ell\in\mathbb{Z}}\hat{F}_V(\ell)e(-\ell x).
  \end{equation}
  We find that
  \begin{align*}
    \hat{F}_V(y)
    &=
      \int_{-\infty}^{\infty}
      \biggl(\frac{\sin 2\pi Vx}{\pi x}\biggr)^2e(-xy)dx
    \\&=
    \int_{-\infty}^{\infty}2V
    \biggl(\frac{\sin \pi 2Vx}{2V\pi x}\biggr)^2e(-(2Vx)y/(2V))d(2Vx)
    \\&=
    2V\widehat{\sinc^2}(y/(2V))=
    2V(1-|y/(2 V)|)^+
  \end{align*}
% see\texttt{https://dsp.stackexchange.com/questions/79257/fourier-transform-of-textrmsinc2100-pi-t}
so that
\begin{align}
  \sum_{m\in\mathbb{Z}}\biggl(\frac{\sin 2\pi V(x-m)}{\pi
    (x-m)}\biggr)^2
  &=\notag
  2V\sum_{|\ell|\le 2V}\biggl(1-\frac{|\ell|}{2 V}\biggr) e(-\ell x)
  \\&=\label{generic0}
  \biggl|\sum_{|\ell|\le L} e(-\ell x)\biggr|^2
\end{align}
the last equality being valid when $V=L+\tfrac12$ where $L$ is an integer. Therefore
\begin{equation}
  \label{eq:13}
  \sum_{m\in\mathbb{Z}}\biggl(\frac{\sin 2\pi V(x-m)}{\pi
    (x-m)}\biggr)^2
  =
  \biggl|
  \frac{\sin 2\pi Lx}{\sin\pi x}
  \biggr|^2.
\end{equation}
This concludes the proof of this lemma.
\end{proof}
%%%%%%%%%%%%%%

%%%%%%%%%%%%%%
\begin{lem}
  \label{dec2}
  When $V-\tfrac12$ is an integer, we have
   \begin{multline*}
    \int_{-1/2}^{1/2}|S(\beta)|^2
    \biggl|\frac{\sin \pi (2V-1)\beta}{\sin\pi\beta}\biggr|^2d\beta
    \ge
    \frac{|S(0)|^2}{N}4V^2\biggl(1-\frac{2V}{N}\biggr)
    \\+
    \int_{-1/2}^{1/2}\biggl|S(\beta)-\frac{S(0)}{N}
    \sum_{-2V< n< N+2V}\mkern-20mu e(n\beta)\biggr|^2
    \biggl|\frac{\sin \pi (2V-1)\beta}{\sin\pi\beta}\biggr|^2d\beta.
  \end{multline*}
\end{lem}
%%%%%%%%%%%%%%

%%%%%%%%%%
\begin{proof}
  Let us set
  \begin{equation*}
    \Sigma
    =
    \int_{-1/2}^{1/2}\biggl|S(\beta)-Z\theta\sum_{-2L\le n\le N+2L}e(n\beta)\biggr|^2
    \biggl|\frac{\sin \pi (2V-1)\beta}{\sin\pi\beta}\biggr|^2d\beta
  \end{equation*}
  for some real $\theta$ and where $Z=S(0)$.
  On expanding the square, we readily find that
  \begin{align}
    \label{splitSigma}
    \Sigma
    =&
    \int_{-1/2}^{1/2}|S(\beta)|^2
    \biggl|\frac{\sin \pi (2V-1)\beta}{\sin\pi\beta}\biggr|^2d\beta
    \\&-2\theta\Re Z\int_{-1/2}^{1/2}\overline{S(\beta)}\sum_{-2L\le n\le N+2L}e(n\beta)
    \biggl|\frac{\sin 2\pi V\beta}{\sin\pi\beta}\biggr|^2d\beta
    \\&+|Z|^2\theta^2
    \int_{-1/2}^{1/2}\biggl|\sum_{-2L\le n\le N+2L}e(n\beta)\biggr|^2
    \biggl|\frac{\sin 2\pi (2V-1)\beta}{\sin\pi\beta}\biggr|^2d\beta.
  \end{align}
  Let us simplify the last two terms.
  Concerning the first of these, we get, with $2V-1=2L$,
  \begin{align*}
    \int_{-1/2}^{1/2}\overline{S(\beta)}&\sum_{-2L\le n\le N+2L}e(n\beta)
    \biggl|\frac{\sin \pi (2V-1)\beta}{\sin\pi\beta}\biggr|^2d\beta
    \\&=
      \sum_{m\le N}\overline{u_m}
      \int_{-1/2}^{1/2}e(-m\beta)\sum_{-2L\le n\le N+2L}e(n\beta)
    \biggl|\sum_{|\ell|\le L} e(-\ell \beta)\biggr|^2d\beta
    \\&=
      \sum_{m\le N}\overline{u_m}
    \sum_{\substack{n\le N,|\ell_1|,|\ell_2|\le L\\
    m=n+\ell_1-\ell_2}}1
    =
      \sum_{m\le N}\overline{u_m}
    \sum_{\substack{n\le N,|h|\le 2L\\
    m=n+h}}(2L+1-|h|).
  \end{align*}
  We find that
  \begin{equation*}
    \sum_{\substack{-2L\le n\le N+2L\\0\le |h|\le 2L\\
        m=n+h}}(2L+1-|h|)
    =
      (2L+1)^2=4V^2,
    \end{equation*}
    since, giving any $m$ and any $h$, a proper $n$ exists.
  The 
  treatment concerning the last term is similar, but with
  $\sum_{-2L\le n\le N+2L}e(n\beta)$ rather than $S(\beta)$. The range
  in $n$ is different, so we have to modify that part when $-2L\le
  n\le 0$ and when $N+1\le n\le 2L$. The total contribution is at most~$(4L+1)V^2$
  so that we
  get
  \begin{equation*}
    |Z|^2\theta^2
    \int_{-1/2}^{1/2}\biggl|\sum_{-2L\le n\le N+2L}e(n\beta)\biggr|^2
    \biggl|\frac{\sin 2\pi (2V-1)\beta}{\sin\pi\beta}\biggr|^2d\beta
    \le |Z|^2\theta^2(N+4L+1)4V^2.
  \end{equation*}
  This leads to the inequality
  \begin{equation}
    \label{lastsplit}
    \int_{-1/2}^{1/2}|S(\beta)|^2
    \biggl|\frac{\sin \pi (2V-1)\beta}{\sin\pi\beta}\biggr|^2d\beta
    \ge
    \Sigma+
    2|Z|^2\theta (4V^2)-|Z|^2\theta^2(N+4L+1)(4V^2).
  \end{equation}
  We choose $\theta=1/N$, and this ends this proof.
\end{proof}
%%%%%%%%%%

%%%%%%%%%%%%%%%%%%%%%%%%%%%%%%%%%%%%%%% 
\section{A special function}
%%%%%%%%%%%%%%%%%%%%%%%%%%%%%%%%%%%%%%%

Let us follow the classical paper \cite{Vaaler*85} by J.~Vaaler.
We first introduce some notation and define
\begin{equation}
  \label{defJhat}
  \hat{J}(t)
  =
  \begin{cases}
    1&\text{if $t=0$},\\
    \pi t(1-|t|)\cot\pi t+|t|&\text{if $0<|t|<1$},\\
    0&\text{if $1\le |t|$}.
  \end{cases}
\end{equation}
As per \cite[Theorem 6]{Vaaler*85},
the function $\hat{J}$ is even, non-negative, continuously
differentiable, and strictly decreasing on~$[0,1]$.
We continue with the definitions from the same paper:
\begin{equation}
  \label{eq:2}
  K(z)=\biggl(\frac{\sin\pi z}{\pi z}\biggr)^2,
  H(z)=\biggl(\frac{\sin\pi z}{\pi }\biggr)^2
  \biggl(\sum_{m\in\mathbb{Z}}\frac{\sgn(m)}{(z-m)^2}
  +\frac{2}{z}\biggr)
  \end{equation}
  (with $\sgn0=1$), and the Beurling-Selberg function
  $B(x)=K(x)+H(x)$. We finally set
  \begin{equation}
    \label{defCabdelta}
  2C_{[a,b],\delta}(x)
  =
  B(\delta(b-x))
  +B(\delta(x-a))\ge2\1_{x\in[a,b]}.
\end{equation}
%%%%%%%%%%%%%
\begin{lem}
  \label{BS}
  Let $\delta>0$ and $M\ge0$ be two parameters.
  The function $C_{[-M,M],\delta}$ is an upper bound for the
  characteristic function of $[-M,M]$.
  When $|t|\le \delta$, we have
  \begin{equation*}
    \hat{C}_{[-M,M],\delta}(t)
    =
    \delta^{-1}(1-|\delta^{-1}t|)\cos 2\pi M t
    +\frac{\hat{J}(\delta^{-1}t)}{\pi t}\sin 2\pi M t.
  \end{equation*}
  When $|t|\ge\delta$, we have $\hat{C}_{[-M,M],\delta}(t)=0$.
  We have $|\hat{C}_{[-M,M],\delta}(t)|\le \hat{C}_{[-M,M],\delta}(0)=2M+\delta^{-1}$.
  %We also have $C_{[-M,M],\delta}(u)\le \1_{[-M,M]}(u)+K(\delta(M-|u|))$.
\end{lem}
%%%%%%%%%%%%%
\noindent
Notice that $\hat{C}_{[-M,M],\delta}(t)=\delta^{-1}\hat{C}_{[-\delta
    M,\delta M],1}(\delta^{-1}t)$.

%%%%%%%%%%%%%
\begin{proof}
  Let us start with somewhat more generality and write
  \begin{equation*}
    2C_{[a,b],\delta}(x)
    =2\1_{x\in[a,b]}
    +B(\delta(b-x))-\sgn\delta(b-x)
    +B(\delta(x-a))-\sgn\delta(x-a).
  \end{equation*}
  Set $F(x)=B(x)-\sgn x$ and $F_2(x)=F(\delta (b-x))$.
  By \cite[Corollary 7]{Vaaler*85}, we have
  \begin{align*}
    \hat{F}(t)
    &=(1-|t|)\1_{|t|\le 1}+\frac{1}{i\pi t}\bigl(\hat{J}(t)-1\bigr)
    % &=
    %   \begin{cases}
    %     1&\text{if $t=0$},\\
    %     1-t+\frac{1-t}{i}\cot\pi t+\frac{1}{i\pi}-\frac{1}{i\pi t}&\text{if $0<t<1$},\\
    %     1+t+\frac{1+t}{i}\cot\pi t-\frac{1}{i\pi}-\frac{1}{i\pi t}&\text{if $-1<t<0$},\\
    %     0&\text{if $1\le |t|$}.
    %   \end{cases}
  \end{align*}
  and classically $\hat{F}_2(t)=e(bt)\delta^{-1}\hat{F}(-t/\delta)$.
  This leads to
  \begin{equation*}
    2\hat{C}_{[a,b],\delta}(t)
    =\frac{e(bt)-e(at)}{i\pi t}
    +\delta^{-1}e(b t)\hat{F}(-\delta^{-1} t)
    +\delta^{-1}e(a t)\hat{F}(\delta^{-1} t).
  \end{equation*}
  Let us proceed by specializing $a=-b=-M$, we get
  \begin{equation*}
     2\delta\hat{C}_{[-M,M],\delta}(t)
    =
    (e(M t)+e(-Mt))(1-|\delta^{-1}t|)^+
    +\frac{e(Mt)-e(-Mt)}{i\pi \delta^{-1}t}
   \hat{J}(\delta^{-1}t).
 \end{equation*}
 Finally, we find that
 \begin{align*}
   \delta\hat{C}_{[-M,M],\delta}(t)
   &=
   (1-|\delta^{-1}t|)^+\cos 2\pi M t
    +\frac{\hat{J}(\delta^{-1}t)}{\pi \delta^{-1}t}\sin 2\pi M t
    \\&=
    \hat{C}_{[-\delta M,\delta M],1}(\delta^{-1}t).
 \end{align*}
 The lemma follows readily.
\end{proof}
%%%%%%%%%%%%%
%%%%%%%%%%%%%
\begin{lem}
  \label{BSSpe}
  Let $\lambda$ be a positive parameter.
  When $|u|\le 1$, we have
  \begin{equation*}
    \hat{C}_{[-\lambda,\lambda],1}(u)
    =
    (1-|u|)\cos 2\pi \lambda u
    +\frac{\hat{J}(u)}{\pi u}\sin 2\pi \lambda u.
  \end{equation*}
  When $|u|\ge 1$, we have $\hat{C}_{[-\lambda,\lambda],1}(u)=0$.
  When $|y|\le \lambda$, we have $C_{[-\lambda,\lambda],1}(y)\ge 1 $.
\end{lem}
%%%%%%%%%%%%%

%%%%%%%%%%%%%%%%%%%%%%%%%%%%%%%%%%%%%%% 
\section{A family of Fourier transforms}
%%%%%%%%%%%%%%%%%%%%%%%%%%%%%%%%%%%%%%%
\label{family}
Let us start with an easy lemma.
%%%%%%%%%%%
\begin{lem}
  \label{fora0}
  Let us set $\myalpha_0(t)=\1_{|t|\le V}$. When $y\in[0,2V]$, we find that
\begin{equation}
  \label{eq:9}
  \int_{-\infty}^\infty \myalpha_0(t)\myalpha_0(t-y)dt
  =\max(0,2V-|y|)=\int_{|y|}^{2V}1dt.
\end{equation}
\end{lem}
%%%%%%%%%%%

We introduce a family of functions in~\eqref{defDhat} to clarify our
process. The main case will be $\myalpha=\myalpha_0$ as defined and
investigated around~\eqref{eq:9}. As it turns out, it is also the one
that governs the proof of Lemma~\ref{Formyalpha}.

%%%%%%%%%%%%%%%%%%%
\begin{lem}
  \label{Formyalpha}
  Let $\myalpha$ be a non-negative even Rieman-integrable function in
  $L^2(\mathbb{R})\cap L^1(\mathbb{R})$ that is non-increasing on
  $[0,\infty)$. The function
  \begin{equation*}
    (a\star a)(y)=
    \int_{-\infty}^{+\infty}\myalpha(t)\myalpha(y-t)dt
  \end{equation*}
  is even and non-increasing on $[0,\infty)$.
\end{lem}
%%%%%%%%%%%%%%%%%%%

%%%%%%%%%%%%
\begin{proof}
  Let us start with the next stream of identities.
  \begin{align*}
    (\myalpha\star \myalpha)(y)
    &=
    \int_{-\infty}^{+\infty}\myalpha(t)\myalpha(t-y)dt
    =
    \int_{-\infty}^{+\infty}\myalpha(t+y)\myalpha(t)dt
    =
    \int_{-\infty}^{+\infty}\myalpha(-t-y)\myalpha(-t)dt
    \\&=
    \int_{-\infty}^{+\infty}\myalpha((-y)-t)\myalpha(t)dt
    =(a\star a)(-y).
  \end{align*}
  This shows that the function $(\myalpha\star \myalpha)$ is indeed even.
  When $\myalpha(t)=\1_{|t|\le T}$, we have
  $(\myalpha\star\myalpha)(y)=(2T-|y|)^+$ which is indeed
  non-increasing when $y\ge0$. This property extends to linear
  combinations of similar functions, provided the coefficients are
  non-negative.
  The proof is complete.
\end{proof}
%%%%%%%%%%%%

%%%%%%%%%%%
\begin{lem}
  \label{Dmyalpha}
  Let $\myalpha$ be a function as in Lemma~\ref{Formyalpha} and let us
  define the non-negative function $\mybeta$ by
  $\int_{-\infty}^{+\infty}\myalpha(t)\myalpha(t-y)dt=\int_{|y|}^\infty
  \mybeta(t)dt$. For $\delta>0$, we consider
  \begin{equation}
    \label{defDhat}
    \hat{D}_{\myalpha,\delta}(u)=\int_{0}^\infty
    \hat{C}_{[-\lambda,\lambda],\delta}(u)
    \mybeta(\lambda)d\lambda
  \end{equation}
  where $C_{[-\lambda,\lambda],\delta}$ is defined in Eq.~\eqref{defCabdelta}.
  We have $\hat{D}_{\myalpha,\delta}(u)=0$ when $|u|\ge\delta$, then
  $\hat{D}_{\myalpha,\delta}(u)\le
  \int_0^\infty(\delta^{-1}+2\lambda)\mybeta(\lambda)d\lambda=\delta^{-1}A(\myalpha)+B(\myalpha)$
  where $A(\myalpha)=\|\myalpha\|_2^2$, $B(\myalpha)=\|\myalpha\|_1^2$ and
  \begin{equation}
    \label{eq:8}
    D_{\myalpha,\delta}(y)
    \ge
    \int_{|y|}^\infty
    \mybeta(\lambda)d\lambda=\int_{-\infty}^{+\infty}\myalpha(t)\myalpha(t-y)dt.
  \end{equation}
\end{lem}
%%%%%%%%%%%

%%%%%%%%%%%%
\begin{proof}
  Most of this lemma is a rather trivial consequence of Lemma~\ref{BS}
  and the properties of the different players. We find at first that
  $B(\myalpha)=2\int_0^\infty\lambda\mybeta(\lambda)d\lambda$. Let us
  express this quantity in terms of $\myalpha$.
  Let us set
  \begin{equation}
    \label{eq:11}
    \gamma(y)=\int_{y}^\infty \mybeta(t)dt.
  \end{equation}
  We have
  \begin{align*}
    B(\myalpha) &=  2\int_0^\infty\lambda\mybeta(\lambda)d\lambda
        = -2 \int_0^\infty\lambda\gamma'(\lambda)d\lambda
    = 2 \int_0^\infty \gamma(\lambda)d\lambda
    \\&=\int_0^\infty \int_{-\infty}^\infty \myalpha(t)\myalpha(t-\lambda)dt\,d\lambda
    +\int_0^\infty \int_{-\infty}^\infty
    \myalpha(t)\myalpha(t+\lambda)dt\,d\lambda
    =\biggl(\int_{-\infty}^\infty \myalpha(t)dt\biggr)^2.
  \end{align*}
  The lemma follows swiftly.
\end{proof}
%%%%%%%%%%%%

%%%%%%%%%%%%%%%%%%%%%%%%%%
\section{Explicit formulae for plots}
%%%%%%%%%%%%%%%%%%%%%%%%%%
Let us more closely the case
$\myalpha_0(t)=\1_{|t|\le 
  V}$. Lemma~\ref{fora0} contains the necessary material and we find that
\begin{equation}
  \label{eq:10}
  A(\myalpha_0)=2V,\quad B(\myalpha_0)=4V^2.
\end{equation}
We have $1/(2Q^2)\le \delta\le 1/(2Q)$ and we may use
\begin{equation}
  \label{eq:1}
  \frac{\hat{D}_{\myalpha_0,\delta}(u)}{2V\delta^{-1}+4V^2}
  \le \1_{[-\delta,\delta]}(u).
\end{equation}
In the main case, which is $\delta$ large, $4V^2$ is much larger than
$2V\delta^{-1}$. We have, on recalling Lemma~\ref{Dmyalpha}, \ref{BS} and~\ref{BSSpe},
\begin{align*}
  \hat{D}_{\myalpha_0,\delta}(u)
  &=\int_0^{2V} \hat{C}_{[-\lambda,\lambda],\delta}(u)d\lambda
  =\delta^{-1}\int_0^{2V} \hat{C}_{[-\delta\lambda,\delta\lambda],1}(\delta^{-1}u)d\lambda
  \\&=\delta^{-1}
  \int_0^{2V} \biggl(
  (1-|\delta^{-1}u|)\cos 2\pi \lambda u
    +\frac{\hat{J}(\delta^{-1}u)}{\pi \delta^{-1}u}\sin 2\pi \lambda u
  \biggr)d\lambda
  \\&=\delta^{-1}
  \biggl(
  (1-|\delta^{-1}u|)\frac{\sin 4\pi Vu}{2\pi u}
    +\frac{\hat{J}(\delta^{-1}u)}{\pi \delta^{-1}u}\frac{1-\cos 4\pi V
  u}{2\pi u}
  \biggr)
\end{align*}
Consequently, we find that
\begin{align}
  \hat{D}_{\myalpha_0,\delta}(u)
  &=\delta^{-2}
  \biggl(
  (1-|w|)\frac{\sin 4\pi W w}{2\pi  w}
    +\frac{\hat{J}(w)}{\pi w}\frac{2\sin^22\pi W
  w}{2\pi w}
  \biggr)
  \notag
  \\&=2\delta^{-1}V
  \bigl(
  (1-|w|)\sinc 4\pi W w
    +2W\hat{J}(w)\sinc^22\pi W w
  \bigr)
  \label{ExactDhata0}
\end{align}
with $w=\delta^{-1}u$ and $W=\delta V$, and where $\sinc u=(\sin
u)/u$ called also the \emph{cardinal sine}.
\begin{figure}[!h]
  \includegraphics[scale=1]{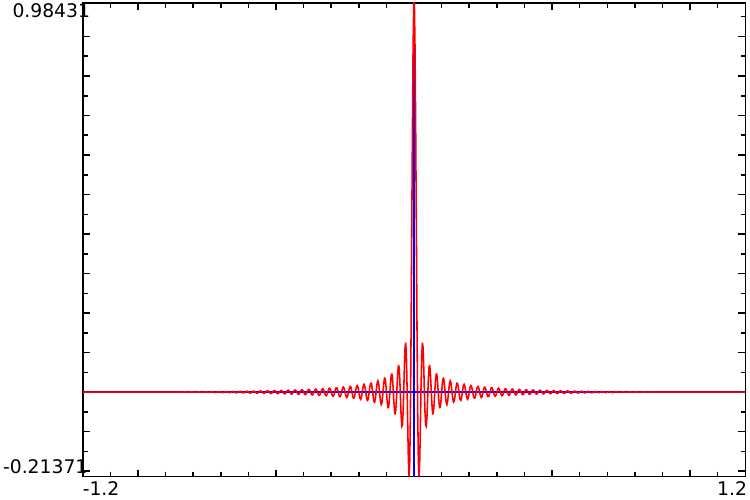}
  \caption{$\hat{D}_{\myalpha_0,\delta}(u)/(4V^2+2\delta^{-1}V)$ with
    $\delta V=40$}
\end{figure}

%%%%%%%%%%%%%%%%%%%%%%%%%%%%%%%%%%%%%%% 
\section{A  Harmonic Analysis divertimento}
%%%%%%%%%%%%%%%%%%%%%%%%%%%%%%%%%%%%%%%
\label{AQuad}
This section is somehow off-topic, but prepares to a problem we will
find later on.
We consider the quadratic form
\begin{equation}
  \label{defQ}
  Q((u_n))=
  \int_{-\delta}^\delta |S(\beta)|^2d\beta
  =\sum_{m,n}u_m\overline{u_n}\frac{\sin (2\pi(m-n)\delta)}{\pi (m-n)}.
\end{equation}
To be in accordance with other works, we assume that $N$ is an
integer.\footnote{Except that D. Slepian considers sequences $(u_n)_{0\le
    n\le N-1}$ while we have $(u_n)_{1\le
    n\le N}$. This explains several of the $-1$ the reader will see in
  the indices.}

%%%%%%%%%%%%%%%%%%%%%
\subsection{Some references on the eigenvectors / eigenvalues of \texorpdfstring{$Q$}{Q}}
%%%%%%%%%%%%%%%%%%%%%
The eigenvectors of this quadratic form are called the \emph{discrete
  prolate spheroidal sequences}, DPSS in short. Following D.~Slepian
in \index{Slepian@Slepian, D.}\cite{Slepian*77}, we write
\begin{equation}
  \label{eq:21}
  \sum_{m}\frac{\sin (2\pi(n-m)\delta)}{\pi (n-m)}
  v^{(k)}_{m-1}(N,\delta)=\lambda_k(N,\delta)v^{(k)}_{n-1}(N,\delta).
\end{equation}
These eigenvalues are normalized so that
\begin{equation}
  \label{eq:22}
  \sum_{n}v^{(k)}_{n-1}(N,\delta)^2=1,
\end{equation}
and
\begin{equation}
  \label{eq:23}
  \sum_{n}v^{(k)}_{n-1}(N,\delta)\ge 0,\quad
  \sum_{n}(N+1-2n)v^{(k)}_{n-1}(N,\delta)\ge 0.
\end{equation}
They are furthermore ordered in $k\in\{0,\cdots, N-1\}$ with
decreasing values of~$\lambda_k(N,\delta)$.
The reader will find numerous papers and numerics on these sequences.
(Voir \cite{Gruenbaum*81})

%%%%%%%%%%%%%%%%%%%%%
\subsection{A minimization problem}
%%%%%%%%%%%%%%%%%%%%%
\label{HA}
As it turns out, the problem will shall meet is to find the best
constant~$C(N,\delta)$ such that
\begin{equation}
  \label{question}
  \int_{-\delta}^\delta |S(\beta)|^2d\beta
  \ge C(N,\delta)|S(0)|^2.
\end{equation}
And we would even be able to restrict this question to sequences
$(u_n)$ that are non-negative, obtaining a constant $C_+(N,\delta)$.
%%%%%%%%%%%%%%%%%
\begin{lem}
  \label{minim}
  Let $(u_n)$ be a sequence of non-negative real numbers. With the
  notation of Lemma~\ref{Dmyalpha}, we have
  \begin{equation*}
    \int_{-\delta}^\delta
    |S(\beta)|^2d\beta
    \ge
    \frac{1}{\delta^{-1}A(\myalpha)+B(\myalpha)}
    \int_{-\infty}^{\infty}
    \Bigl|\sum_m u_m\myalpha(t-m)\Bigr|^2dt
  \end{equation*}
  where $S(\beta)=\sum_{n\le N}u_n e(n\beta)$.
  We also have
  \begin{equation*}
  N\int_{-\delta}^\delta
    |S(\beta)|^2d\beta
    \ge
    \biggl(1-\frac{2}{1+\sqrt{\delta N+1}}\biggr)
    |S(0)|^2.
  \end{equation*}
\end{lem}
%%%%%%%%%%%%%%%%
A usual question in Harmonic Analysis is to show that the Fourier
transform cannot be too concentrated, and this is linked with the
Uncertainty Principle. It is not the place to detail such a huge area,
and we simply refer the reader to two papers, the first one~\cite{Logvinenko-Sereda*74} by
V.\,N.~Logvinenko \& Ju.\,F.~Sereda
and the second one~\cite{Nazarov*92} by F.\,L.~Nazarov.
We failed to find references concerning~$C(N,\delta)$ or~$C^+(N,\delta)$.

%%%%%%%%%%%%%%
\begin{proof}
    We use
  Lemma~\ref{Dmyalpha}.
  which simplifies in:
  \begin{equation*}
    \frac{\hat{D}_{\myalpha,\delta}(\beta)}{\delta^{-1}A(\myalpha)+B(\myalpha)}
    \le
    \1_{[-\delta,\delta]}(\beta)
  \end{equation*}
  This gives us
  \begin{align*}
    \int_{-\delta}^{\delta}|
    S(\beta)
    |^2d\beta
    \ge&
    \int_{-\delta}^{\delta}|S(\beta)|^2
    \frac{\hat{D}_{\myalpha,\delta}(\beta)d\beta}{\delta^{-1}A(\myalpha)+B(\myalpha)}
    .
  \end{align*}
  We may extend the integral over $\mathbb{R}$ by positivity obtaining
\begin{equation*}
  \int_{-\delta}^{\delta}|S(\beta)|^2d\beta\ge
    \frac{1}{\delta^{-1}A(\myalpha)+B(\myalpha)}
    \sum_{m,n}u_m\overline{u_n}
    D_{\myalpha,\delta}(m-n)
  \end{equation*}
  on using Lemma~\ref{Dmyalpha}. As $(u_n)$ is assumed to be
  non-negative, we may bound below $D_{\myalpha,\delta_q}(m-n)$ by
  $\int_{-\infty}^{+\infty}\myalpha (t)\myalpha(t-y)dt$, getting
  \begin{equation*}
    \int_{-\delta}^{\delta}|S(\beta)|^2d\beta
    \ge
    \frac{1}{\delta^{-1}A(\myalpha)+B(\myalpha)}
    \int_{-\infty}^{\infty}
    \Bigl|\sum_m u_m\myalpha(t-m)\Bigr|^2dt.
  \end{equation*}
  This establishes the first part of the lemma. Let us now select
  $\myalpha=\myalpha_0$ and use Lemma~\ref{integraldecomposition}. We
  infer that, since $\Sigma^2=B(\myalpha)$,
  \begin{equation*}
    N\int_{-\delta}^{\delta}|S(\beta)|^2d\beta
    \ge
    \frac{B(\myalpha_0)}{\delta^{-1}A(\myalpha_0)+B(\myalpha_0)}
    \biggl(1-\frac{2V}{N}\biggr)|S(0)|^2
    \ge
    \frac{2\delta V}{1+2\delta V}
    \biggl(1-\frac{2\delta V}{\delta N}\biggr)|S(0)|^2.
  \end{equation*}
  With $v=2\delta V$ and $w=\delta N$, we readily find that the optimal
  optimal choice for $V$ is given by $v^2+2v-w=0$,
  i.e. $v=-1+\sqrt{1+w}$. The resulting constant is
  \begin{equation*}
    \frac{v(w-v)}{(v+1)w}
    =
    \frac{v^2(v+1)}{(v+1)w}
    =
    \frac{w+2-2\sqrt{w+1}}{w}
    =
    1-\frac{2}{1+\sqrt{w+1}}.
  \end{equation*}
  The proof is complete.
\end{proof}
%%%%%%%%%%%%%%

In the notation above, the last part of Lemma~\ref{minim} implies that
  \begin{equation}
    \label{minCplus}
    C_+(N,\delta)\ge \biggl(1-\frac{2}{\sqrt{\delta N}}\biggr)/N.
  \end{equation}
  It is easy to show that (when $1/2\ge \delta$)
  \begin{equation}
  \label{eq:24}
  1/N\ge C(N,\delta)\ge C_+(N,\delta).
\end{equation}
This is simply proved by noticing that
\begin{equation*}
  N =
  \int_{-1/2}^{1/2} \,\Bigl|\sum_n e(n\beta)\Bigr|^2d\beta
  \ge
  \int_{-\delta}^\delta \,\Bigl|\sum_n e(n\beta)\Bigr|^2d\beta
  \ge C(N,\delta) \Bigl|\,\sum_n 1\,\Bigr|^2.
\end{equation*}
%%%%%%%%%%%%%%%%%%%%%
\subsection{Numerical optimization with positivity conditions}
%%%%%%%%%%%%%%%%%%%%%
We may consider this problem with Kuhn-Tucker conditions, i.e. examine
the function
\begin{equation}
  \label{eq:16}
  H_{N,\delta}((u_n),\lambda)
  =2\delta\sum_{m,n}u_mu_n\sinc (2\pi(m-n)\delta)
  -\sum_{n}\lambda_nu_n
  +\mu\biggl(\sum_{n}u_n-1\biggr).
\end{equation}
For every local minima of $2\delta\sum_{m,n}u_mu_n\sinc (2\pi(m-n)\delta)$
under the conditions $-u_n\le 0$ for every~$n$ and $\sum_{n}u_n=1$,
there exists\footnote{In our present problem, this necessary condition
  can furthermore be proved to be sufficient.}
for every $n$ some $\lambda_n\ge0$ and a real $\mu$ such that
\begin{equation}
  \label{eq:18}
  \left\{
    \begin{aligned}
      &\forall m,\quad
      2\delta\sum_{n}u_n\sinc (2\pi(m-n)\delta)= \lambda_m -\mu,
      \\
      &
        \forall m,\quad
        u_m\ge0\quad\text{and}\quad \lambda_m u_m=0,
      \\
      &
      \sum_{n}u_n=1.
    \end{aligned}
  \right.
\end{equation}
In which case the minimum is readily checked to be~$-\mu$.
This means that there exists a subset $I\subset\{1,\cdots,N\}$ of
indices that are such that $u_m=0$ and that $\lambda_m=0$ when
$m\notin I$. Since in each case, we get a linear system of $N-|I|+1$ unknowns
(the $+1$ is for the $\mu$-variable), with $N-|I|+1$ equations.
Such a system is likely to be non-degenerate (its determinant is a
real analytic function of $\delta$, so it cannot vanish on any
non-trivial interval without vanishing everywhere; this implies that
we need only move $\delta$ by a tiny amount to get a non-zero
determinant).

We get the next table for $NC_+(N,\delta)$:
\begin{tabular}{|c|c|c|}
  \hline
  $N$&$\delta=0.2$&$\delta=0.1$\\
  \hline
  10&0.8488...&0.7105...\\
  11&0.8593...&0.7333...\\
  12&0.8567...&0.7474...\\
  13&0.8608...&0.7471...\\
  14&0.8794...&0.7458...\\
  15&0.8882...&0.7426...\\
  16&0.8896...&0.7563...\\
  17&0.8852...&0.7773...\\
  18&0.8973...&0.7943...\\
  19&0.9063...&0.8063...\\
  20&0.9098...&0.8143...\\
  \hline
\end{tabular}

\smallskip

And a table for $NC_+(N,\delta)/(1-2/\sqrt{\delta N})$:
\smallskip

{\small
  % \begin{tabular}{|c|c|c|c|c|c|}
  % \hline
  % $N$&14&15&16&17&18\\
  % \hline
  % $\delta=0.3$&38.7932...&16.3200...&10.8103...&8.2590...&6.7719...\\
  % \hline
  % \end{tabular}

 \hfill \begin{tabular}{|c|c|c|c|c|}
  \hline
  $N$&19&20&21&22\\
  \hline
  $\delta=0.3$&5.8499...&5.1885...21&4.6818...&4.3113...\\
  \hline
        \end{tabular}    
\par}
\bigskip

Sadly enough, this last data is barely significant as $N\delta$ is not
large enough. A major slowing factor is that we have to investigate
every possible subset of indices.
One may remark that when an admissible optimum (i.e. a minimum with
non-negative coordinates) has been found for a given subset, it is not
required to investigate the subsets of this subset. This remark does
not speed the process, as it turns out that checking that a given subset 
does not belong to the list of cleared ones takes too much time.

%\clearpage

\bigskip
%\noindent
Here is the Pari-GP script used to compute the above tables.

\vspace*{-4pt}
{\small
\begin{verbatim}
{getmin(listindices, delta) =
   my(locN = length(listindices), mymat = matrix(locN + 1),
      vecimage = vector(locN + 1), res, isok = 1, firstnonzeron = 1);

   \\ Initialization:
   for(m = 1, locN, 
      for(n = m, locN,
         mymat[m, n] = sinc(2*Pi*(listindices[m]-listindices[n])*delta)*2*delta;
         mymat[n, m] = mymat[m, n]));
   for(m = 1, locN, mymat[m, locN + 1] = 1);
   for(n = 1, locN, mymat[locN + 1, n] = 1);

   for(n = 1, locN, vecimage[n] = 0);
   vecimage[locN + 1] = 1;

   \\ Get possible minimum:
   res = matsolve(mymat, vecimage~);

   \\ Analyze the result:
   while(res[firstnonzeron] == 0, firstnonzeron++);
   if(res[firstnonzeron] < 0, res = -res);
   for(n = firstnonzeron + 1, locN,
      if(res[n] < 0, isok = 0; break,));
   return([isok, - res[locN + 1]]);}
   
{work(N, delta, DoTell = 1) =
   my(locres, res = 1);
   forsubset(N, listindices,
      if(length(listindices) == 0,,
         locres = getmin(listindices, delta);
         if(locres[1] == 1,
            if(DoTell == 1,
              print("Minimum at ", listindices, " = ", locres[2]),);
            res = min(locres[2], res),)));
   return(res);}
   
\end{verbatim}

  \par}

%%%%%%%%%%%%%%%%%%%%%%%%%%
\section{Base Camp, proof of Theorem~\ref{WeightedSieve}}
%%%%%%%%%%%%%%%%%%%%%%%%%%
Here is the first key point of the method we propose.
%%%%%%%%%%%
\begin{lem}
  \label{BaseCamp}
  The function $\myalpha$ being as in Lemma~\ref{Formyalpha}, and the
  hypotheses of Theorem~\ref{WeightedSieve} being met, we have
  \begin{multline*}
    \biggl(1-\frac{2V}{N}\biggr)
    \sum_{\substack{q\le Q}}
    \frac{g(q)}{1+q(q+Q)A(\myalpha)/B(\myalpha)}
    \Bigr|\sum_m u_m\Bigl|^2/N
    \\\le
    \sum_{\substack{q\le Q}}
    \frac{g(q)}{B(\myalpha)+q(q+Q)A(\myalpha)}
    \int_{-\infty}^\infty\Bigl|\sum_{m}u_m\myalpha(t-m)\Bigr|^2dt
    \le \sum_{n}|u_n|^2.
  \end{multline*}
\end{lem}
%%%%%%%%%%%
This inequality leads to the classical Brun-Titchmarsh inequality, i.e. with
the factor $2+o(1)$ if we are to choose~$V=N/\log N$,
$Q=\sqrt{V}/\log V$. The novelty is that we have expressed the
error term as an integral of a positive quantity over a set of density~1.

%%%%%%%%%%%%%%%%%
\begin{proof}
  We start from Lemma~\ref{Step2b} and set $\delta_q^{-1}=q(q+Q)$.
  We find, by using Lemma~\ref{L2LB}, that:
  \begin{equation*}
    \sum_{a\mode q}
    \int_{-\delta_q}^{\delta_q}\biggl|
    S\biggl(\frac{a}{q}+\beta\biggr)
    \biggr|^2
    d\beta
    \ge
      \int_{-\delta_q}^{\delta_q}| S(\beta)
    |^2
    d\beta.
  \end{equation*}
  As $(u_n)$ is assumed to be
  non-negative, Lemma~\ref{minim} applies to this quantity, getting
  \begin{equation*}
    \sum_{a\mode q}
    \int_{-\delta_q}^{\delta_q}\biggl|
    S\biggl(\frac{a}{q}+\beta\biggr)
    \biggr|^2
    d\beta
    \ge
    \frac{g(q)}{\delta_q^{-1}A(\myalpha)+B(\myalpha)}
    \int_{-\infty}^{\infty}
    \bigl|\sum_m u_m\myalpha(t-m)\bigr|^2dt.
  \end{equation*}
  We conclude this segment of the proof by appealing to
  Lemma~\ref{integraldecomposition} and on noticing that
  $\Sigma^2=B(\myalpha)$. We sum over $q\le Q$ and appeal to
  Lemma~\ref{Step2b} to bound above the left-hand side.  Our lemma
  readily follows.
\end{proof}
%%%%%%%%%%%%%%%%%

%%%%%%%%%%%%%%
\begin{proof}[Proof of Theorem~\ref{WeightedSieve}]
  We simply apply Lemma~\ref{BaseCamp} to $\myalpha=\myalpha_0$ and recall~\eqref{eq:10}.
\end{proof}
%%%%%%%%%%%%%%

%%%%%%%%%%%%%%%%%% 
\section{Using the localization}
%%%%%%%%%%%%%%%%%%
%%%%%%%%%%%
\begin{lem}
  \label{integraldecompositionBis}
  Let $(u_m)_{m\le N}$ be a sequence of complex numbers. We have 
  \begin{multline*}
    \frac{N}{4V^2(1-2V/N)}\int_{-\infty}^\infty
    \Bigl|\sum_{|m-t|\le V}u_m\Bigr|^2
    dt
    \\
    \ge
    |S(0)|^2
    +
    \biggl|S(\beta)
    -
    \frac{S(0)}{N}\int_{-V}^{N+V}\mkern-15mu e(\beta t)dt
    +\Ocal^*(2\pi\beta V S(0))\biggr|^2.
  \end{multline*}
\end{lem}
%%%%%%%%%%%

%%%%%%%%%%%%%%
\begin{proof}
Let us notice that $S(0)=\sum_m u_m=\sum_m |u_m|$. We set
%%%%%%%%%%%%%
\begin{equation}
  \Delta(N,V)
  =\int_{-V}^{N+V}
  \biggl|\sum_{|m-t|\le V}u_m-\frac{2VS(0)}{N}\biggr|^2
  dt.
\end{equation}  
%%%%%%%%%%%%% 
We swiftly obtain:
%%%%%%%%%%%%% 
\begin{align*}
  \Delta(N,V)    
  &=
    \int_{-V}^{N+V}
    \biggl|\sum_{|m-t|\le V}u_m
    \bigl(e((m-t)\beta)+\Ocal^*(2\pi\beta V)\bigr)-\frac{2V S(0)}{N}\biggr|^2
  dt
  \\&=
    \int_{-V}^{N+V}
    \biggl|\sum_{|m-t|\le V}u_m
    \bigl(e(m\beta)+\Ocal^*(2\pi\beta V)\bigr)-\frac{2V S(0)e(\beta t)}{N}\biggr|^2
  dt
\end{align*}
We may then use Cauchy's inequality (in reverse) to infer that
\begin{align*}
(N+2V)\Delta(N,V)&\ge
  \biggl|\int_{-V}^{N+V}
  \biggl(
    \sum_{|m-t|\le V}u_m
    \bigl(e(m\beta)+\Ocal^*(2\pi\beta V)\bigr)-\frac{2V S(0)e(\beta
  t)}{N}\biggr)dt\biggr|^2
  \\&\ge
  \biggl|2V\sum_{m}u_m
  \bigl(e(m\beta)+\Ocal^*(2\pi\beta V)\bigr)
  -
  \frac{2V S(0)}{N}\int_{-V}^{N+V}e(\beta
  t)dt\biggr|^2
  \\&\ge
  \biggl|2VS(\beta)+
  \Ocal^*(2\pi\beta V^2S(0))
  -
  \frac{2V S(0)}{N}\int_{-V}^{N+V}e(\beta
  t)dt\biggr|^2
  \\&\ge
  4V^2\biggl|S(\beta)
  -
  \frac{S(0)}{N}\biggl(\int_{-V}^{N+V}
  e(\beta t)dt+\Ocal^*(2\pi\beta V N)\biggr)\biggr|^2.
\end{align*}
%%%%%%%%%%%%%
We finally only have to plug this inequality in
Lemma~\ref{integraldecomposition}, and to simplify the expression to get
the present lemma.
\end{proof}
%%%%%%%%%%%%%%%

%%%%%%%%%%
\begin{proof}[Proof of Theorem~\ref{conditional1}]
  We start from Theorem~\ref{WeightedSieve} and minorize the left-hand
  side through Lemma~\ref{integraldecompositionBis}. We obtain, when $V<N/2$,
  \begin{equation*}
    |S(0)|^2
    +
    \biggl|S(\beta)
    -
    \frac{S(0)}{N}\int_{-V}^{N+V}\mkern-15mu e(\beta t)dt
    +\Ocal^*(2\pi\beta V S(0))\biggr|^2
    \le \frac{N\|S\|_2^2}{2V(1-2V/N)L^*(Q,V)}.
  \end{equation*}
  Shuffling the terms around gives the next bound
  \begin{equation*}
    \biggl|\frac{S(\beta)}{S(0)}
    -
    \int_{-V}^{N+V}\mkern-15mu e(\beta t)\frac{dt}{N}
    \biggr|
    \le \sqrt{\frac{N\|S\|_2^2/S(0)^2}{2V(1-2V/N)L^*(Q,V)}-1}
    +2\pi |\beta| V.
  \end{equation*}
  In practice, the modification from $L$ of~\eqref{defL} to $L^*$
  of~\eqref{defLstar} allows to raise~$Q$ from~$\sqrt{V}/(\log V)^A$
  for some positive~$A$ to some multiple of~$\sqrt{V}$. But since~$V$
  is going to already be of the shape $N/(\log N)^A$, since saving is
  often quantitative only. It is clearer to use the bound
  \begin{equation}
    \label{eq:15}
    2V L^*(Q,V)\ge \frac{2V}{2V+2Q^2}L(Q).
  \end{equation}
  This leads to the bound
  \begin{equation*}
    \biggl|\frac{S(\beta)}{S(0)}
    -
    \int_{-V}^{N+V}\mkern-15mu e(\beta t)\frac{dt}{N}
    \biggr|
    \le \sqrt{\frac{1+Q^2/V}{1-2V/N}\frac{N\|S\|_2^2/L(Q)}{S(0)^2}-1}
    +2\pi |\beta| V.
  \end{equation*}
  This inequality will be useful when $S(0)^2$ is close
  to~$N\|S\|_2^2/L(Q)$. So let us continue with the mild hypothesis of
  the theorem, namely
  \begin{equation*}
    2S(0)^2\ge N\|S\|_2^2/L(Q).
  \end{equation*}
  On using $\sqrt{a+b}\le \sqrt{a}+\sqrt{b}$ for non-negative $a$ and
  $b$, and on assuming further that $V\le 4N$ to simplify the
  outcome,
  we reach
  \begin{equation*}
    \biggl|\frac{S(\beta)}{S(0)}
    -
    \int_{-V}^{N+V}\mkern-15mu e(\beta t)\frac{dt}{N}
    \biggr|
    \le \sqrt{\frac{N\|S\|_2^2/L(Q)}{S(0)^2}-1}
    +\frac{2Q}{\sqrt{V}}+2\pi |\beta| V.
  \end{equation*}
  This ends the proof.
\end{proof}
%%%%%%%%%%

%%%%%%%%%%%%%
\begin{proof}[Proof of Corollary~\ref{Manyps}]
  We find classically, in this case, we have $L(Q)\ge \log Q$, so we may use this value. We
  select $V=N/(\log N)^{B+1}$  and $Q=\sqrt{N}/(\log N)^{\frac12B+1}$. We find
  that
  \begin{equation*}
    \frac{N}{L}
    \le
    \frac{2}{\log N -(B+2)\log\log N}
    .
  \end{equation*}
  Notice that $Z\ge N/\log N$ provided that $N$ is large enough in
  terms of~$C$. On adopting an obvious notation and with $y=(\log\log
  N)/\log N$, Theorem~\ref{conditional1} gives us
  \begin{align*}
    \biggl|\frac{S(\beta)}{Z}
    -
    \int_{-V}^{N+V}\mkern-15mu e(\beta t)\frac{dt}{N}
    \biggr|
    &\le
      \sqrt{\frac{1+Cy}{1-(B+2)y}-1}
      +\frac{2}{\sqrt{\log N}}+\frac{7 |\beta|N}{(\log N)^{B+1}}
    \\&\le
      \sqrt{\frac{(B+2+C)y}{1-(B+2)y}}
      +\frac{2}{\sqrt{\log N}}+\frac{7}{\log N}.
  \end{align*}
  We deduce from this inequality that
  \begin{equation*}
    \biggl|\frac{S(\beta)}{Z}
    -
    \int_{0}^{N} e(\beta t)\frac{dt}{N}
    \biggr|
    \le 2 \sqrt{\frac{(B+C+1)\log\log N}{\log N}}
  \end{equation*}
  provided that $1-(B+2)y\ge 1/2$ and $N$ is large enough
  (independently of the choice of~$C$ and/or of~$A$)
  The reader will swiftly complete the proof from there.
\end{proof}
%%%%%%%%%%%%%

%%%%%%%%%%%%%%%
\begin{proof}[Proof of Theorem~\ref{Nearby}]
  We start from Lemma~\ref{BaseCamp} with $\myalpha=\myalpha_0$ and
  express the integral by Lemma~\ref{Switch} to get
  \begin{equation*}
    \frac{L^*}{2V}\int_{-1/2}^{1/2}|S(\beta)|^2
    \biggl|\frac{\sin \pi (2V-1)\beta}{\sin\pi\beta}\biggr|^2d\beta
    \le S(0)^2.
  \end{equation*}
  We finally appeal to Lemma~\ref{dec2}.
\end{proof}
%%%%%%%%%%%%%%%

%%%%%%%%%%%%%
\begin{proof}[Proof of Corollary~\ref{ManypsL20}]
  Let us select $V=[N/(\log N)^A]$ and $Q^2=N/(\log N)^{2A}$ in Theorem~\ref{Nearby} and use for
  $u_n$ the characteristic function of the set we are counting.
  The large sieve gives us the a-priori bound $Z\ll N/(\log
  N)^\kappa$, and therefore
  \begin{equation*}
    \frac{|S(0)|^2}{N}\frac{2V}{N}\ll \frac{Z}{(\log N)^{\kappa+A}}.
  \end{equation*}
  Furthermore, we have
  \begin{equation*}
    2VL^*(Q,V)\ge L(Q)\frac{2V}{2V+2Q^2}
  \end{equation*}
  so that we reach
  \begin{multline*}
    \int_{-(\log N)^A/N}^{(\log N)^A/N}
    \biggl|S(\beta)-\frac{Z}{N}
    \sum_{-2V< n< N+2V}\mkern-20mu e(n\beta)\biggr|^2d\beta
    \\
    \le \biggl(\frac{N}{L(Q)}-Z
    \biggr)\frac{Z}{N}+\Ocal\biggl(\frac{Z}{(\log N)^{\kappa+A}}\biggr).
  \end{multline*}
  We next shorten the summation in $n$ by appealing to the inequality
  \begin{equation*}
    \biggl|S(\beta)-\frac{Z}{N}
    \sum_{n\le N} e(n\beta)\biggr|^2
    \le
    2\biggl|S(\beta)-\frac{Z}{N}
    \sum_{-2V< n< N+2V}\mkern-20mu e(n\beta)\biggr|^2
    +
    2\biggl|\frac{Z}{N}
    \sum_{\substack{-2V< n< N+2V\\ n\notin[1,N]}}\mkern-20mu e(n\beta)\biggr|^2.
  \end{equation*}
  Consequently, we get
  \begin{multline*}
    \int_{-(\log N)^A/N}^{(\log N)^A/N}
    \biggl|S(\beta)-\frac{Z}{N}
    \sum_{ n\le N}e(n\beta)\biggr|^2d\beta
    \\
    \le 2\biggl(\frac{N}{L(Q)}-Z
    \biggr)\frac{Z}{N}+
    \Ocal\biggl(\frac{Z}{(\log N)^{\kappa+A}}
    +
    \frac{Z^2V}{N^2}
    \biggr).
  \end{multline*}
  We may replace $Q$ by $\sqrt{N}$ by noticing that
  \begin{equation*}
    \frac{N}{L(Q)}-\frac{N}{L(\sqrt{N})}=\frac{N(\log\log N)}{(\log N)^{\kappa+1}}
  \end{equation*}
  from which the lemma follows readily.
\end{proof}
%%%%%%%%%%%%%

%%%%%%%%%%%%%%%%%% 
\section{Additional notes}
%%%%%%%%%%%%%%%%%%
The optimality of the present proof may be
questioned, at three stations: the local lower bound provided by
Theorem~\ref{L2S}, the lower bound in Lemma~\ref{Step2b} and the replacement
of the characteristic function $\1_{[-\delta_q,\delta_q]}(\alpha-a/q)$
of the arc
$[a/q-\delta_q,a/q+\delta_q]$ by the function
$F_q(\alpha-a/q)=\hat{D}_{[0,N],\delta_q}(\alpha-\frac{a}{q})$ which has a
sharp peak at $\alpha=a/q$.
The
circle method philosophy relies (in most cases) on the fact that the
trigonometric polynomial $S(\alpha)$ has its maximum at rational
points, and this for varying values of~$N$. Such a regularity should lead to a
sharp peak at~$\alpha=a/q$. This argument goes in favour of saying
that the replacement of $\1_{[-\delta_q,\delta_q]}(\alpha-a/q)$ by
$F_q(\alpha-a/q)$ does not loose much. And then we may believe that the
lower bound of Lemma~\ref{Step2b} does not loose much either.

As noted earlier, Theorem~\ref{L2S} may be thought as optimal in its
form: when we use the Cauchy's inequality on the
\emph{equality~\eqref{L1S}}, the result is an equality when~$S(a/q)$
and~$\hat{\psi}^*_q(a)$ are proportional and therefore a \emph{plausible} set of
values for~$S(a/q)$ is available that takes care of the support condition.
This being said, Figures~\ref{test1png} and~\ref{test2png} seem to say that the lower
bound in Theorem~\ref{L2S} is often not sharp enough  when the
modulus~$q$ becomes of the size~$\sqrt{N}$. It also hints that, if we
are to rely on a single modulus, then Theorem~\ref{L2S} will be hard
to improve, as the minimum seems to increase, but very slowly.

The behaviour for a single~$q$ but varying length~$N$ is also
relevant; the plots we obtain show a very stable situation.
%%%%%%%%%%%%
\begin{figure}[!h]
  \includegraphics[scale=0.6]{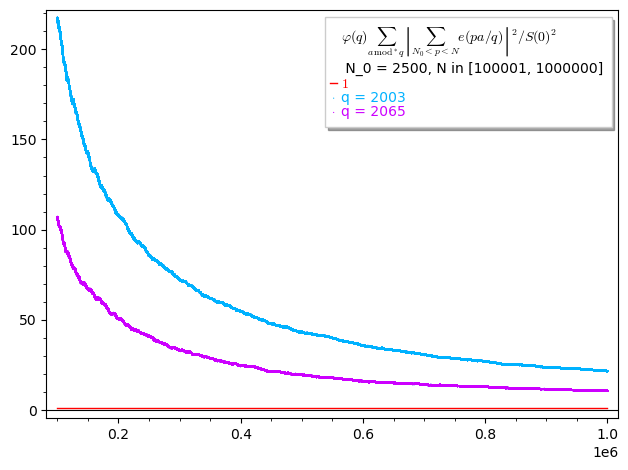}
  \caption{A single modulus but varying length, in the prime case}
  \label{follow1png}
\end{figure}
%%%%%%%%%%%%%

We supplement these figures with the same plot but concerning the
twin primes from the initial segment, see Figures~\ref{test1twinpng}
and~\ref{test2twinpng}. The appearance of separate families of
moduli where the quantity
\begin{equation*}
  \frac{\mu^2(q)}{\prod_{p|q,p>2}2/(p-2)}
  \sum_{a\mode q}|S(a/q)|^2/|S(0)|^2
  \qquad\biggl(S(\alpha)=\sum_{\substack{p\le N\\ \text{$p+2$ prime}}}e(p\alpha)\biggr)
\end{equation*}
have a different behaviour occurs clearly. The number of prime factors
is the deciding quantity.

%%%%%%%%%%%%
\begin{figure}[!h]
  \includegraphics[scale=0.6]{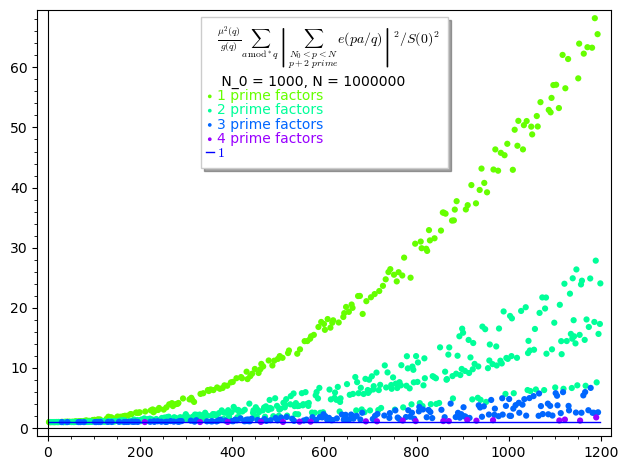}
  \caption{On the optimality of Theorem~\ref{L2S} for the twin primes}
  \label{test1twinpng}
\end{figure}
%%%%%%%%%%%%%
%%%%%%%%%%%%
\begin{figure}[!h]
  \includegraphics[scale=0.6]{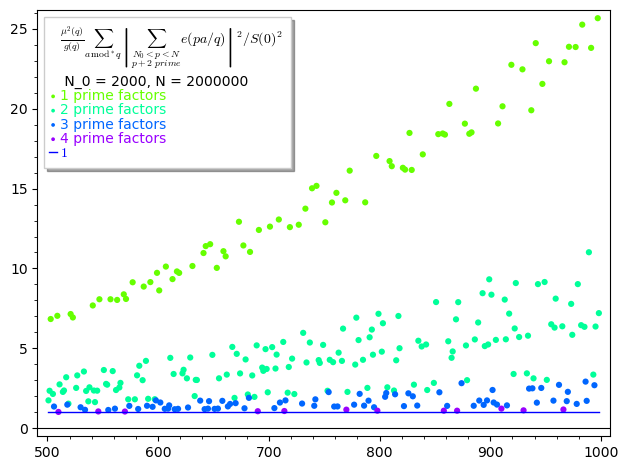}
  \caption{On the optimality of Theorem~\ref{L2S} for the twin primes}
  \label{test2twinpng}
\end{figure}
%%%%%%%%%%%%%

% \bibliographystyle{authordate1}
\bibliographystyle{plain}
%\bibliography{Local.bib}

\end{document}